\RequirePackage{fix-cm} 
\documentclass[a4paper,twoside,reqno,11pt]{amsart}

\usepackage{a4wide,verbatim,orcidlink}
\usepackage{amsaddr}
\usepackage[english]{babel}
\usepackage[latin1]{inputenc}
\usepackage[T1]{fontenc}
\usepackage{epsfig,rotating,color}
\usepackage{amsmath,amsfonts,amssymb,amsgen,amsbsy,amsthm,leftindex}
\usepackage{eucal,esint,dsfont,mathrsfs}
\usepackage{stmaryrd} 

\usepackage[square,numbers]{natbib}
\newcommand{\eqdef}{\stackrel{\text{\tiny{def}}}{=}} 

\newcommand{\ud}{\mathrm{d}}

\newcommand{\ui}{\mathrm{i}}

\newtheorem{thm}{Theorem}

\newcommand{\half}{{\textstyle{1\over2}}}

\newlength{\intwidth}

\usepackage[OT2,T1]{fontenc}
\DeclareSymbolFont{cyrletters}{OT2}{wncyr}{m}{n}
\DeclareMathSymbol{\Sha}{\mathalpha}{cyrletters}{"58}

\renewcommand{\oint}{\ointctrclockwise}

\title[]{\bf Explicit solution for the hyperbolic homogeneous scalar 
one-dimensional conservation law}

\author[D. Clamond]{Didier CLAMOND}

\address{Universit\'e C\^ote d'Azur, CNRS UMR 7351, 
Laboratoire J. A. Dieudonn\'e,\\ Parc Valrose, F-06108 Nice cedex 2, France.}
\email{didier.clamond@univ-cotedazur.fr}

\date{\today}

\begin{document}
\maketitle

\begin{abstract}
A complex integral formula provides an explicit solution of the initial 
value problem for the nonlinear scalar 1D equation $u_t+[f(u)]_x=0$, for any 
flux $f(u)$ and initial condition $u_0(x)$ that are analytic. This   
formula is valid at least as long as $u$ remains analytic. 

\bigskip
\noindent \textbf{Keywords:} hyperbolic conservation law; nonlinear wave equation; 
explicit solution.

\smallskip
\noindent \textbf{MSC:} {35C15; 35F25; 35L60; 35L65; 30E20.}

\bigskip\bigskip

\end{abstract}

\section{Introduction}\label{Introduction} %

A classical hyperbolic partial differential equation is the one-dimensional 
first-order scalar homogeneous conservation law \citep{Dafermos2016}
\begin{equation} \label{NLWE}
u_t\/+\/\left[\,f(u)\,\right]_x\,=\,0,
\end{equation}
where $u(x,t)$ is the unknown conserved quantity (often a scalar velocity field or a 
density of mass) and $f(u)$ is the flux, the function $f$ being given. Special cases 
include the transport equation when $f(u)=c\/u$, the Hopf (aka inviscid Burgers) equation 
when $f(u)=\half\/u^2$, the Buckley--Leverett equation \citep{BuckleyLeverett1942}, 
various traffic flow equations, and many other more or less well-known equations. 
A typical Cauchy problem is the resolution of \eqref{NLWE} for $x\in\mathds{R}$, $t>0$ 
with the given initial condition $u(x,0)=u_0(x)$. 

For a differentiable flux $f(u)$, denoting $c(u)=\ud f(u)/\ud\/u$, 
the initial value problem can be rewritten 
\begin{equation}\label{IVP}
u_t\/+\/c(u)\,u_x\,=\,0, \qquad u(x,0)\,=\,u_0(x), 
\end{equation}
that is, in particular, a typical model equation for gas dynamics. 
It has the very well-known solution in implicit form \citep{Whitham1974} 
\begin{equation}\label{solimpIVP}
u(x,t)\,=\,u_0\!\left(X\right), \qquad 
X\/\eqdef\/x\/-\/t\,c(u(x,t)).
\end{equation}
Obviously, the celerity $c(u)$ is also given implicitly by the relation 
\begin{equation}\label{solcuimp}
c\!\left(u(x,t)\right)\/=\,c\!\left(u_0(X)\right).
\end{equation} 
Equation \eqref{NLWE} being ubiquitous (in mathematics, physics, etc.),  
it is of theoretical and practical interest to derive an explicit 
solution, at least for some type of solutions. It seems that such a general 
explicit solution never appeared before in the literature, so it is the 
subject of the present note. 

It turns out that, for sufficiently regular $f$ and $u_0(x)$, the solution 
$u(x,t)$ can be written explicitly as the formula \eqref{eqint2} below 
(involving a contour integral in the $z$-plane, complex extension of the 
real $x$-axis). In this short note, we actually give two explicit solutions, 
one for $c(u)$ in section \ref{secsolcu} and one for $u$ in section \ref{secsolu}.

\section{Explicit solution for $c(u)$}\label{secsolcu}

The first exact explicit solution is easily obtained for $c(u)=\ud f(u)/\ud\/u$, 
as expressed in the theorem \ref{thmcu} below.
\begin{thm}\label{thmcu}
If $u_0(x)$ and $c(u)=f'(u)$ are both real-analytic, and (at least) as long as 
$c(u(x,t))$ remains analytic, the initial value problem 
\eqref{NLWE} has the explicit solution for $c(u)$ 
\begin{equation}
c\!\left(u(x,t)\right)=\,\frac{1}{2\/\ui\/\pi\/t}\oint_\gamma\log\!
\left[\/1\/+\/\frac{t\,c\!\left(u_0\!\left(z\right)\right)}{z-x}\right]\ud\/z,
\label{solcuint}
\end{equation}
where $\gamma$ is a rectifiable Jordan curve surrounding $x$ such that $c\circ u_0$ 
is holomorphic inside $\gamma$ and $|z-x|>|t\,c(u_0(z))|$ on $\gamma$. 
\end{thm}
\begin{proof} 
Introducing $X\eqdef 
x-t\,c\!\left(u(x,t)\right)$ and $F(X)\eqdef-t\,c\!\left(u_0\!\left(X\right)\right)$,  
relation \eqref{solcuimp} yields $X=x+F(X)$. The latter has explicit solution for $X$ 
given by the integral Lagrange reversion theorem \eqref{solyfint} (c.f. Appendix 
\ref{applres}), yielding at once formula \eqref{solcuint}. 
\end{proof}

Note that, multiplying \eqref{IVP} by $c'(u)$, one gets \citep{Whitham1974}
\begin{equation}
\left[\,c(u)\,\right]_t\/+\left[\,\half\,c(u)^2\,\right]_x\,=\,0,\label{eqadvc}
\end{equation}
that is a Hopf equation for $c(u)$. Thus, if $c(u)=u$, formula \eqref{solcuint} 
provides an explicit solution for $u$. If $c(u)\neq u$ and if $c$ is invertible, 
an explicit solution for $u$ is trivially obtained from \eqref{solcuint}.
However, when $c(u)\neq u$, $c$ may be not invertible, or $c^{-1}$ may be unknown 
or too complicated. Therefore, when $c(u)\neq u$, it is of practical interest to 
derive another explicit solution for $u$.

\section{Explicit solution for $u(x,t)$}\label{secsolu}

The second exact explicit solution is for $u(x,t)$,      
given by the theorem \ref{thmu} below.
  
\begin{thm}\label{thmu}
If $f(u)$ and $u_0(x)$ are both real-analytic, and at least as long 
as $u(x,t)$ remains real-analytic, an explicit solution of the initial 
value problem \eqref{NLWE} is
\begin{align}
u(x,t)\,=\,\frac{1}{2\/\ui\/\pi}\oint_{\gamma}\frac{1\/+\/t\,c'(u_0(z))\,u_0'(z)}
{z\/-\/x\/+\/t\,c(u_0(z))}\,u_0(z)\,\ud\/z,\label{eqint2}
\end{align}
with $u_0'(z)\eqdef\ud u_0(z)/\ud\/z$,  $c(u)\eqdef\ud f(u)/\ud u$, 
$c'(u)\eqdef\ud c(u)/\ud u$ and where $\gamma$ 
is a rectifiable Jordan curve surrounding $x$ such that $u_0(z)$ is holomorphic inside 
$\gamma$ and $|z-x|>|t\,c(u_0(z))|$ on $\gamma$. 
\end{thm}
\begin{proof}
The condition of real-analyticity means that there exist, for all $x\in\mathds{R}$, 
a circle of centre $x$ and radius $R>0$ where $u_0(z)$ and $c(u_0(z))$ are both analytic. 
The condition $|z-x|>|t\,c(u_0(z))|$ on $\gamma$ is necessary for having poles 
inside $\gamma$, otherwise the integral is zero. 

For $t=0$, invoking the residue theorem \citep{Goursat1916}, 
formula \eqref{eqint2} yields  
\begin{equation}
u(x,0)\,=\,\frac{1}{2\/\ui\/\pi}\oint_{\gamma}\frac{u_0(z)\,\ud\/z}
{z\/-\/x}\,=\,u_0(x),
\end{equation}
the existence of $\gamma$ being guaranteed by the real-analyticity of $u_0$, the  
condition $|z-x|>0$ on $\gamma$ bringing no further restrictions when $t=0$.
For $t\neq0$, after differentiation, integration by parts and some elementary 
algebra, \eqref{eqint2} yields 
\begin{align}
u_x(x,t)\,&=\,\frac{1}{2\/\ui\/\pi}\oint_\gamma\frac{u_0'(z)\,\ud\/z}
{z-x+t\,c(u_0(z))},
\qquad
u_t(x,t)\,=\,\frac{\ui}{2\/\pi}\oint_\gamma\frac{c(u_0(z))\,u_0'(z)\,
\ud\/z}{z-x+t\,c(u_0(z))},
\end{align}
and with the change of dummy variable $Z\eqdef z+t\,c(u_0(z))$ --- so $\gamma
\mapsto\Gamma$, $u_0(z)\mapsto U_0(Z)\eqdef u_0(z(Z))$ and $c(u_0(z))\mapsto 
c(U_0(Z))$
--- one gets (exploiting the residue theorem)
\begin{align}
u_x\/=\/\frac{1}{2\/\ui\/\pi}\oint_\Gamma\frac{U_0'(Z)\,\ud\/Z}{Z-x}\,=\,U_0'(x), 
\quad
u_t\/=\/\frac{\ui}{2\/\pi}\oint_\Gamma\frac{c(U_0(Z))\,U_0'(Z)\,\ud\/Z}{Z-x}\/
=\/-\/c(U_0(x))\,U_0'(x). \label{uxut}
\end{align}
Moreover, from the integral Lagrange reversion formula \eqref{solyfint} and the 
theorem \ref{thmcu}, we have
\begin{equation}
Z\,=\,x\quad\implies\quad z\,=\,x\/+\/\frac{\ui}{2\/\pi}\oint_\gamma
\log\!\left[\/1\/+\/\frac{t\,c(u_0(\zeta))}{\zeta-x}\right]\ud\/\zeta\,
=\,x\/-\/t\,c\!\left(u(x,t)\right),
\end{equation}
hence $U_0(x)=u_0(x\/-\/t\,c\!\left(u(x,t)\right))=u(x,t)$ and then $c(U_0(x))
=c\!\left(u(x,t)\right)$. Therefore, the relations \eqref{uxut} 
yield $u_t+c(u)u_x=\left[\,c(u(x,t))\/-\/c(U_0(x))\,\right]U_0'(x)=0$.
\end{proof}

Note that formula \eqref{eqint2} is a special form of the so-called {\em argument 
principle\/} \citep{Beardon2020}. Formula \eqref{solcuint} is also  a special form 
of the argument principle, as one can easily verify via one integration by parts 
and some elementary algebra.

\section{Discussion}

To the author knowledge, formula \eqref{eqint2} is the first explicit solution of 
equation \eqref{NLWE} for quite general $f(u)$ and $u_0(x)$ (though {\em a priori\/} 
restricted to 
$f$ and $u_0$ being real-analytic). Note that the so-called {\em explicit\/} 
Lax--Ole\u{\i}nik formula \citep{Coclite2024} is actually only semi-explicit because 
it requires the resolution of an optimisation problem involving the Legendre 
transform of the flux $f(u)$.  

The proof above is valid if there exists a contour $\gamma$ enclosing $x-t\,c(u_0(z))$ 
within the analyticity region of $u_0$. This condition of existence 
is satisfied as long as $u(x,t)$ remains real-analytic. The explicit solution 
\eqref{eqint2} is thus valid for real-analytic initial condition $u_0(x)$ and flux 
$f(u)$, as long as $u(x,t)$ remains real-analytic. 
However, it is well-known that, in general, singularities appear in finite time 
\citep{Dafermos2016}. 

Whether or not the solution \eqref{eqint2} can be used (or generalised) 
to also describe some weak solutions is the subject of current investigations; 
there is room for hope in this direction, for at least three reasons. 
First, theorems \ref{thmcu} and \ref{thmu} should remain valid if ``real-analytic'' 
is replaced by ``piecewise real-analytic''; the question on how to join the pieces 
being open. Second, as shown in \eqref{eqint2}, the explicit solution is 
obtained from a generalised argument principle, a ``principle'' applicable 
with several zeros and poles in the integrand \citep{Beardon2020,Goursat1916}. 
Third, recently, a relation somehow similar to \eqref{eqint2} was derived by 
\citet{Gerard2023} 
for the Benjamin--Ono equation $u_t+uu_x=\epsilon\/\mathscr{H}u_{xx}$ ($\mathscr{H}$ 
the Hilbert transform). In the dispersionless limiting case $\epsilon\to0$, this 
equation reduces to the inviscid Burgers equation, and \citet{Gerard2024} shows that 
his explicit formula remains valid in this limiting case.

\appendix

\section{Residue and Lagrange reversion theorems}\label{applres}

The residue theorem \citep{Goursat1916}, for any function $F$ analytic inside a 
contour $\gamma$ enclosing $x$ and continuous on $\gamma$, is
\begin{equation}\label{resfor}
\frac{\ud^n\/F(x)}{\ud\/x^n}\ =\ \frac{n!}{2\/\ui\/\pi}\oint_\gamma
\frac{F(z)\,\ud\/z}{(z-x)^{n+1}} \qquad \forall n\in\mathds{N}_0.
\end{equation}
Note that: {\em i}) if $x$ is outside $\gamma$, then the integral is zero; 
{\em ii}) if $x$ is on $\gamma$, then the integral must be interpreted in 
the sense of Cauchy's principal value, and $2\pi$ must be replaced by the 
value of the inner angle at $x$ (e.g., $\pi$ where the contour is smooth).

According to the Lagrange reversion theorem \citep{Goursat1916}, the implicit equation 
$y=x+F(y)$ with $F$ an analytic function, has explicit solution
\begin{align}\label{lagrev}
y(x)\,&=\,x\/+\/\sum_{n=1}^\infty\frac{1}{n!}\/\frac{\ud^{n-1}\/F(x)^n}{\ud\/x^{n-1}}
\,=\,y(0)\/+\/\int_0^x\sum_{n=0}^\infty\frac{1}{n!}\/\frac{\ud^{n}\/F(\xi)^n}
{\ud\/\xi^{n}}\,\ud\/\xi,
\end{align}  
the second equality being of practical interest if the integral can be obtained in 
closed-form.
From the second equality \eqref{lagrev}, the residue theorem yields
\begin{align}
y(x)\,&=\,y(0)\/+\/\int_0^x\oint_\gamma\,\sum_{n=0}^\infty\left[\frac{F(z)}{z-\xi}
\right]^n\/\frac{\ud\/z}{z-\xi}\,\frac{\ud\/\xi}{2\/\ui\/\pi}\,=\,y(0)\/+\/
\oint_\gamma\int_0^x\frac{\ud\/\xi}{z-\xi-F(z)}\,\frac{\ud\/z}{2\/\ui\/\pi}\nonumber\\
&=\,\frac{\ui}{2\/\pi}\oint_\gamma\log\!\left[\/z-x-F(z)\/\right]\ud\/z\,=\,
x\/+\/\frac{\ui}{2\/\pi}\oint_\gamma\log\!\left[\/1\/-\/\frac{F(z)}{z-x}\right]\ud\/z, 
\label{solyfint} 
\end{align} 
provided that $|z-x|>|F(z)|$ everywhere on $\gamma$. The last equality in \eqref{solyfint} 
can also be obtained directly applying a similar summation to the first equality of 
\eqref{lagrev}. 
Formula \eqref{solyfint} is called here the {\em integral Lagrange reversion theorem\/} 
\citep[\S51]{Goursat1916} 
(it can also be derived directly from the generalised argument principle 
\citep[\S48, eq. 40]{Goursat1916}).

\addcontentsline{toc}{section}{References}
\bibliographystyle{abbrvnat}

\begin{thebibliography}{8}
\providecommand{\natexlab}[1]{#1}
\providecommand{\url}[1]{\texttt{#1}}
\expandafter\ifx\csname urlstyle\endcsname\relax
  \providecommand{\doi}[1]{doi: #1}\else
  \providecommand{\doi}{doi: \begingroup \urlstyle{rm}\Url}\fi

\bibitem[Beardon(2020)]{Beardon2020}
A.~F. Beardon.
\newblock \emph{Complex Analysis: The Argument Principle in Analysis and
  Topology}.
\newblock Dover, 2020.

\bibitem[Buckley and Leverett(1942)]{BuckleyLeverett1942}
S.~E. Buckley and M.~C. Leverett.
\newblock Mechanism of fluid displacements in sands.
\newblock \emph{Trans. AIME}, 146:\penalty0 107--116, 1942.

\bibitem[Coclite(2024)]{Coclite2024}
G.~M. Coclite.
\newblock \emph{Scalar Conservation Laws}.
\newblock SpringerBriefs in Mathematics. Springer, 2024.

\bibitem[Dafermos(2016)]{Dafermos2016}
M.~Dafermos, C.
\newblock \emph{Hyperbolic Conservation Laws in Continuum Physics}, volume 325
  of \emph{A Series of Comprehensive Studies in Mathematics}.
\newblock Springer, 4th edition, 2016.

\bibitem[G\'erard(2023)]{Gerard2023}
P.~G\'erard.
\newblock An explicit formula for the {B}enjamin--{O}no equation.
\newblock \emph{Tunisian J. Math.}, 5\penalty0 (3):\penalty0 593--603, 2023.

\bibitem[G\'erard(2024)]{Gerard2024}
P.~G\'erard.
\newblock The zero dispersion limit for the {B}enjamin--{O}no equation on the
  line.
\newblock \emph{Comptes R. Math.}, 363:\penalty0 619--634, 2024.

\bibitem[Goursat(1916)]{Goursat1916}
E.~Goursat.
\newblock \emph{A course in Mathematical Analysis}, volume II, part I.
\newblock Ginn \& Co., 1916.

\bibitem[Whitham(1974)]{Whitham1974}
G.~B. Whitham.
\newblock \emph{Linear and nonlinear waves}.
\newblock Wiley, 1974.

\end{thebibliography}

\end{document}